\documentclass[12pt,twoside,draft]{cmpart}
\newcommand{\N}{{\mathbb{N}}}
\newcommand{\eps}{\varepsilon}

\author{Zavarzina O.O.}
\title{On expand-contract plasticity in quasi-metric spaces}

\shorttitle{Plasticity in quasi-metric spaces}

\institute{Department of Mathematics and Informatics, V.N. Karazin Kharkiv  National University, 61022 Kharkiv, Ukraine}

\email{olesia.zavarzina@yahoo.com}

\enabstract {Zavarzina O.O.}%
{On expand-contract plasticity in quasi-metric spaces}%
{It is known that if any function acting from precompact metric space to itself increases the distance between some pair of points then it must decrease distance between some other pair of points. We show that this is not the case for quasi-metric spaces. After that, we present some sufficient conditions under which the previous property holds true for hereditarily precompact quasi-metric spaces.}
{non-expansive map, plastic space, quasi-metric space}

\thanks{The author is grateful to her scientific advisor Vladimir Kadets for constant help with this project. The research was partially supported  by the  National Research Foundation of Ukraine funded by Ukrainian State budget in frames of project 2020.02/0096 ``Operators in infinite-dimensional spaces:  the interplay between geometry, algebra and topology''. The author was also partially supported by the Volkswagen Foundation grant within the
frameworks of the international project ``From Modeling and Analysis to Approximation'' and Universities for Ukraine (U4U) Non-Residential Fellowship Program.}


\UDC{?????}

\received{17.01.2023}  

\begin{document}

\maketitle


\section*{Introduction}
Let $X$ be a nonempty set. By a quasi-metric on $X$, we mean a nonnegative
real-valued function $d$ on $X\times X$ such that for $x,y,z\in X$  we have
\begin{itemize}
  \item[(a)] $d(x; x) = 0$,
  \item[(b)] $d(x; z) \leq d(x; y) + d(y; z)$,
  \item[(c)] $d(x; y) = 0 \Rightarrow x = y$.
\end{itemize}
The set  $X$ with a fixed  quasi-metric is called a quasi-metric space.

The concept of quasi-metric is a generalization of the concept of metric. The difference consists in the absence of condition  $d(x; y) = d(y, x)$, that is why a  quasi-metric is informally called an `` asymmetric metric''. The natural question what theorems involving metric spaces remain valid for quasi-metric ones is a natural and popular line of research. An important for us example is \cite{KMrReilVam}, where some questions of convergence and precompactness are regarded from this point of view. More about the theory of quasi-metric spaces, especially about the  asymmetric version of normed spaces and adaptation of functional analysis to the asymmetric case one can find in the monograph \cite{Cob}.

The conjugate of a quasi-metric $d$ on $X$,
denoted by $d^{-1}$, is a quasimetric on $X$, defined by $d^{-1}(x,y) = d(y, x)$ for
$x,y\in X$.
The set
$$B(x,\varepsilon) = \{y\in X : d(x,y) < \varepsilon\}$$
is the d-ball with center $x$ and radius $\varepsilon$.
A map $F: M \to M$  is called non-expansive, if $d(F(x), F(y)) \leq d(x,y)$ for all $x,y \in M$.
\begin{definition}
The quasi-metric space $M$ is called expand-contract  plastic (or simply, an EC-space) if every non-expansive bijection from $M$ onto  itself is an isometry.
\end{definition}

For {\em metric}  spaces, Naimpally, Piotrowski, and Wingler proved the following theorem.
\begin{theorem}[{\cite[Theorem 1.1]{NaiPioWing}}] \label{NPW}
Let $(X, d)$ be a totally bounded metric space, and let $f\colon X\to X$ be a function.
If there exist points $p, q\in X$ such that $d(f (p), f (q)) > d(p, q)$, then there exist points
$r, s\in X$ such that $d(f (r), f (s)) < d(r, s)$.
\end{theorem}
This theorem immediately  implies that every precompact metric space is expand-contract  plastic.

Although the definition might seem simple, in many concrete metric spaces that are not totally bounded the problem of establishing their  plasticity is unexpectedly difficult. For example, it is unknown whether the unit ball of every Banach space is  expand-contract  plastic, moreover the question remains open for such basic spaces as $c_0$, $C[0, 1]$ or $L_1[0, 1]$, more about this problem and partial positive results one can find in recent articles \cite{AnKaZa},\cite{CKOW2016},\cite{HLZ},\cite{KZ2017},\cite{KaZa},\cite{Zav},\cite{Zav2}. In \cite{AnKaZa} some generalizations of Theorem \ref{NPW} to uniform spaces were considered. Unfortunately, these generalizations are not applicable to quasi-metric spaces because uniform spaces do not generalize quasi-metric ones. The corresponding generalization of quasi-metric spaces are quasi-uniform spaces, which differ from the uniform ones by absence of the symmetry axiom (i.e. in a quasi-uniform space $(X, \mathfrak{U})$ the condition $U\in \mathfrak{U}$ does not necessarily imply $U^{-1}\in \mathfrak{U}$.)

The aim of this article is to study the possibility of generalizing Theorem \ref{NPW} to quasi-metric spaces. We demonstrate that the direct generalization does not work, but the theorem is valid under mild additional conditions.

Analogously to metric spaces, a quasi-metric space $(X,d)$ is called precompact if for every $\varepsilon>0$ there is a finite subset $Y \subset X$ such that $X\subseteq \cup_{x\in Y}B(x,\varepsilon)$. Surprisingly, this definition loses an important property of ordinary precompact spaces: a subspace of a precompact  quasi-metric space is not necessarily precompact. This motivates the following definition.

\begin{definition}
The quasi-metric space $(X, d)$ is said to be hereditarily precompact if every subspace of $(X, d)$ is precompact
\end{definition}
  Obviously, hereditary precompactness implies precompactness. Recall also another important definition.
\begin{definition}
 A sequence $\{x_n\}_{n\in\N}$ in a quasi-metric space $(X, d)$ is called a left (resp. right) K-Cauchy sequence if for any given $\varepsilon > 0$, there is an integer
$N \in\N$ such that $d(x_n, x_m) < \varepsilon$ $($resp. $d(x_m, x_n) < \varepsilon$$)$ for all $m \geq n \geq N$.
\end{definition}
Theorem 3 from \cite{KMrReilVam} states that space $(X, d)$ is hereditarily precompact  if and only if every infinite sequence of points in $X$ has a left K-Cauchy infinite subsequence.


\section{Main results}
First of all, the statement of Theorem \ref{NPW} is false in quasi-metric spaces. In the further exposition we will use the following notations:
$$2\N=\{2k, k\in \N\},$$
$$2\N-1=\{2k-1, k\in \N\}.$$

\begin{theorem}\label{examp}
There exist a hereditarily precompact quasi-metric space  $(X, d)$ and a non-contractive mapping $F \colon X \to X$ which is not an isometry.
\end{theorem}
\begin{proof}
Take $X = \{e^{in}\}_{n=1}^{\infty}$ equipped with the following quasi-metric $d$:
\begin{equation*}
d(e^{in},e^{im})=
\begin{cases}
  |e^{in}-e^{im}| \text{ if }  m,n \in 2\N, m\geq n,\\
  \quad \quad \quad  \quad \text{ or }  m,n\in 2\N-1, m\geq n;\\
 2 \text{ if } n=1, m\in 2\N,\\
  3 \text{ otherwise}.
 \end{cases}
\end{equation*}
Let us check the axioms of quasi-metric. Items (a) and (c) are obviously satisfied. It remains to verify the triangle inequality
\begin{equation}\label{triangle_in}
 d(e^{in},e^{im})\leq d(e^{in},e^{ik})+d(e^{ik},e^{im})
\end{equation}
 in all possible cases. First of all, let us note that $|e^{in}-e^{im}|\leq |e^{in}|+|e^{im}|=2$  and the triangle inequality holds for any three points of the space.

Now we are going to show the precompactness of $X$. Observe that $X=\{e^{2ki}\}_{k=1}^{\infty}\cup\{e^{(2k-1)i}\}_{k=1}^{\infty}$. We will construct a finite $\eps$-net for $\{e^{2ki}\}_{k=1}^{\infty}$. The procedure of constructing for $\{e^{(2k-1)i}\}_{k=1}^{\infty}$  is the same. The union of these two $\eps$-nets will be a finite $\eps$-net for $X$.

 Obviously, the set $\{e^{2ki}\}_{k=1}^{\infty}$ with the usual metric $\rho(x,y) = |x - y|$  is precompact. So, in the usual metric for every $\eps>0$ there is a finite $\eps$-net  $\{e^{2k_ni}\}_{n=1}^N$. Let us consider the usual metric balls $B_\rho(e^{2k_ni},\eps)$. There are points in $B_\rho(e^{2k_ni},\eps)$ which belong to the corresponding quasi-metric balls $B(e^{2k_ni},\eps)$  and there are points which do not. Those points $e^{2mi} \in B_\rho(e^{2k_ni},\eps)$ that do not belong to  $B(e^{2k_ni},\eps)$ must satisfy the condition $2m < 2k_n$, so there are only finitely many of them.
Consequently, there are only finitely many points in
$$
\bigcup_{n=1}^N B_\rho(e^{2k_ni},\eps)\setminus \bigcup_{n=1}^N  B(e^{2k_ni},\eps).
$$
All these points together with $\{e^{2k_ni}\}_{n=1}^N$ will serve as a finite $\eps$-net for $\{e^{2ki}\}_{k=1}^{\infty}$.
 In the same way one may construct a finite $\eps$-net for any subspace of $X$, so $X$ is hereditarily precompact.

The required $F \colon X \to X$ will be the following shift mapping: $F(e^{in})=F(e^{i(n+2)})$.
Let us show that $F$ is non-contractive. We will consider three cases.
\begin{enumerate}
  \item If $d(F(e^{in}),F(e^{im})) =3$ there is nothing to check, since 3 is the biggest possible distance in the space.
  \item Due to definition of $F$ there are no points in $X$ such that $d(F(e^{in}),F(e^{im})) = 2$.
  \item $d(F(e^{in}),F(e^{im}))= |F(e^{in})-F(e^{im})|$, means that $n,m\in 2\N-1$ or $n,m \in 2\N$, $m>n$. Since addition of 2 does not change the parity, we have
  \begin{align*}
  d(F(e^{in}),F(e^{im}))=d(e^{i(n+2)},e^{i(m+2)})=|e^{i(n+2)}-e^{i(m+2)}|\\
  = |e^{2i}||e^{in}-e^{im}|= |e^{in}-e^{im}|=d(e^{in},e^{im}).
  \end{align*}
\end{enumerate}
Finally,  let us demonstrate that $F$ is not an isometry.
$$d(F(e^{i}),F(e^{2i}))=d(e^{3i},e^{4i})=3>2=d(e^{i},e^{2i}).$$
\end{proof}

Further we are going to  present some additional conditions which can save the quasi-metric version of Theorem \ref{NPW}.

\begin{theorem}
Let $(X,d)$ and $(X,d^{-1})$ be hereditarily precompact quasi-metric spaces. Let $F\colon X\to X$ be a function. If there are points $p$ and $q$  in $X$ such that $d(F(p),F(q))>d(p,q)$ then there are points $r, s\in X$ such that $d(F(r),F(s))<d(r,s)$.
\end{theorem}
\begin{proof}
Let us argue by contradiction. Suppose $d(F(x),F(y))\geq d(x,y)$ for all $x,y \in X$. Let us introduce the following  auxiliary metric:
$$ d_s(x,y)=d(x,y)+d(y,x).$$
 The corresponding metric space $(X, d_s)$ is precompact because for any sequence in $(X, d_s)$ there is a Cauchy subsequence (due to item (c) of Theorems 3 of \cite{KMrReilVam} from any sequence in $(X,d)$  one may extract a left K-Cauchy subsequence and then  extract a right K-Cauchy subsequence from this left K-Cauchy subsequence using the same item of the same theorem but for $(X,d^{-1})$). Corollary 1.2 from \cite{NaiPioWing} together with our hypothesis imply that $F$ is a $d_s$-isometry, i.e. $d_s(F(x),F(y))= d_s(x,y)$ for all $x,y\in X$. In particular,
$d_s(F(p),F(q))= d_s(p,q)$. So, we have two conditions: $d(F(p),F(q))+d(F(q),F(p))=d(p,q)+d(q,p)$ and $d(F(p),F(q))>d(p,q)$. These conditions imply the inequality $d(F(q),F(p))<d(q,p)$, which contradicts our assumption.
\end{proof}

\begin{corollary}
Let $(X,d)$ be a hereditarily precompact quasi-metric space satisfying the following condition: for every $\eps>0$ there is $\delta>0$ such that
\begin{equation}\label{convcond}
d(x, y) < \delta \Rightarrow d(y, x)< \eps
\end{equation}
for any $x,y \in X$. Let $F\colon X\to X$ be a function. Then the existence of points $p, q \in X$ with $d(F(p),F(q))>d(p,q)$ implies the existence of  $r, s\in X$ with $d(F(r),F(s))<d(r,s)$.
\end{corollary}
\begin{proof}
The existence of $\eps$-net in $(X,d)$ for every $\eps>0$ and condition \ref{convcond} implies the existence of $\eps$-net in $(X,d^{-1})$ for every $\eps>0$. So $(X,d^{-1})$ is hereditarily precompact as well and it remains to apply the previous theorem to finish the proof.
\end{proof}

\begin{corollary}
Let $(X,d)$ be a hereditarily precompact quasi-metric space. Let there is a constant $C>0$, such that $d(x,y)\leq Cd(y,x)$ for all $x,y \in X$. Let $f \colon X\to X$ be a function. If there are points $p$ and $q$ such that $d(f(p),f(q))>d(p,q)$ then there are points $r$ and $s$ such that $d(f(r),f(s))<d(r,s)$.
\end{corollary}
\begin{proof}
For every $\eps>0$ there exists $\delta=\frac{\eps}{C}$ such that
$$d(x,y)< \delta \Rightarrow d(y,x)\leq Cd(x,y)<C\delta=\eps.$$
\end{proof}

We tried to find an analogue of {example from the proof of Theorem \ref{examp}} with a bijective function $F$, but surprisingly failed. So, the following question remains open.

Question. Is it true that for every precompact quasi-metric space $(X, d)$ every {\em bijective} non-contractive mapping  $F\colon X\to X$ is an isometry?


\begin{thebibliography}{10}
\bibitem{AnKaZa}  Angosto C., Kadets V., Zavarzina O. \emph{Non-expansive bijections, uniformities and polyhedral faces}, J. Math. Anal. Appl. 2019, \textbf{471}(1-2), 38--52.
\bibitem{CKOW2016} Cascales B.,  Kadets V., Orihuela J.,  Wingler E. J. \emph{Plasticity of the unit ball of a strictly convex Banach space}, Revista de la Real Academia de Ciencias Exactas, F\'{\i}sicas y Naturales. Serie A. Matem\'aticas 2016,  \textbf{ 110}(2), 723--727.
 \bibitem{Cob}  Cobzas S. Functional Analysis in Asymmetric Normed Spaces, Birkh\"auser, 2013.
\bibitem{HLZ}	Haller R., Leo N., Zavarzina O. \emph{Two new examples of Banach spaces with a plastic unit ball}, Acta et Commentationes Universitatis Tartuensis de Mathematica 2022, \textbf{26} (1), 89--101.
 \bibitem{KZ2017} Kadets V., Zavarzina O. \emph{Non-expansive bijections to the unit ball of $\ell_1$-sum of strictly convex Banach spaces}, Bull. Aust. Math. Soc. 2018, \textbf{97}(2), 285--292.
\bibitem{KaZa} Karpenko I., Zavarzina O. \emph{Linear expand-contract plasticity of ellipsoids revisited}, Matematychni Studii 2022, \textbf{57}(2), 192--201.
\bibitem{KMrReilVam}  K\"unzi H. P. A., Mr\v sevi\'c M., Reilly I. L., Vamanamurthy M. K. \emph{Convergence, precompactness
and symmetry in quasi-uniform spaces}, Math. Japon. 1993, \textbf{38}, 239--253.

\bibitem{NaiPioWing}  Naimpally S. A., Piotrowski Z., Wingler E. J. \emph{Plasticity in metric spaces}, J. Math. Anal. Appl. 2006, \textbf{313}, 38--48.
\bibitem{Zav}Zavarzina O. \emph{Non-expansive bijections between unit balls of Banach spaces}, Annals of Functional Analysis 2018, \textbf{9}(2), 271--281.
\bibitem{Zav2} Zavarzina O. \emph{Linear expand-contract plasticity of ellipsoids in separable Hilbert spaces}, Matematychni Studii 2019, \textbf{51}(1), 86--91.
\end{thebibliography}
\end{document}